\documentclass[leqno,12pt]{article}

\usepackage{amsmath,amssymb,amsthm}
\usepackage{thmtools}
\usepackage{bm} 
\usepackage{ascmac} 
\usepackage[capitalize]{cleveref} 
\usepackage{comment} 
\usepackage{stmaryrd} 
\usepackage{graphicx}
\usepackage{url}
\usepackage{tikz-cd} 

\theoremstyle{definition}

\newtheorem{thm}{Theorem}[section]

\newtheorem{crl}[thm]{Corollary}
\newtheorem{prp}[thm]{Proposition}
\newtheorem{lmm}[thm]{Lemma}
\newtheorem{rmk}[thm]{Remark}
\newtheorem{dfn}[thm]{Definition}
\newtheorem{exm}[thm]{Example}

\newtheorem*{thm*}{Theorem}
\newtheorem*{cnj*}{Conjecture}

\newcommand{\LR}{\Leftrightarrow}

\newcommand{\NN}{\mathbb N}
\newcommand{\ZZ}{\mathbb Z}
\newcommand{\QQ}{\mathbb Q}

\newcommand{\CC}{\mathbb C}

\makeatletter
\@addtoreset{equation}{section}

\makeatother

\title{
A note on the second supplementary law of rational power residue symbols
}
\author{Yoshinosuke HIRAKAWA\thanks{Graduate School of Sciences and Technology for Innovation, Yamaguchi University, 1677-1, Yoshida, Yamaguchi, 753-8512, Japan,
\protect\url{yhirakawa@yamaguchi-u.ac.jp}} 
\and Tomokazu KASHIO\thanks{Department of Mathematics, Faculty of Science and Technology, Tokyo University of Science, Noda, Chiba 278-8510, Japan,
\protect\url{tomokazu_kashio@rs.tus.ac.jp}} 
\and Ryutaro SEKIGAWA\thanks{Division of Science,
School of Science and Engineering, 
Tokyo Denki University, 
Hatoyama, Saitama, 350-0394, Japan, 
\protect\url{sekigawa_ryutaro@mail.dendai.ac.jp}} 
\and Naoaki TAKADA\thanks{\protect\url{6122517@alumni.tus.ac.jp}}
\and Shuji YAMAMOTO\thanks{Interdisciplinary Faculty of Science and Engineering, Shimane University,
1060 Nishi-Kawatsu, Matsue, Shimane, 690-8504, Japan, \protect\url{yamashu@riko.shimane-u.ac.jp}}}

\makeatletter
\newcommand{\subjclass}[2][2020]{%
  \let\@oldtitle\@title%
  \gdef\@title{\@oldtitle\footnotetext{#1 \emph{Mathematics subject classification(s).} #2}}%
}
\newcommand{\keywords}[1]{%
  \let\@@oldtitle\@title%
  \gdef\@title{\@@oldtitle\footnotetext{\emph{Key words and phrases.} #1.}}%
}
\makeatother


\begin{document}

\maketitle

\begin{abstract}
As a natural generalization of the Legendre symbol, the $q$-th power residue symbol $(a/p)_q$ is defined for primes $p$ and $q$ with $p\equiv 1 \bmod q$. 
In this paper, we generalize the second supplementary law by providing an explicit condition for $(q/p)_q = 1$, when $p$ has a special form $p = \sum_{i=0}^{q-1} m^i n^{q-1-i}$.
This condition is expressed in terms of the polylogarithm $\mathrm{Li}_{1-q}(x)$ of negative index.
Our proof relies on an argument similar to Lemmermeyer's proof of Euler's conjectures for cubic residue.
\end{abstract}

\section{Main results} \label{s1}

The rational $q$-th power residue symbol is defined by 
\begin{align*}
\left( \frac{a}{p} \right)_q:=
\begin{cases}
1 & \text{if there exists $b\in \ZZ$ satisfying $a \equiv b^q \bmod p$}, \\
-1 & \text{otherwise},
\end{cases}
\end{align*}
where $a\in \ZZ$, $q  \in \NN$ with $q \geq 2$ and $p$ is a prime with $p\equiv 1 \mod q $.
This paper deals with the case where $q$ is a prime number.
When $q =2$, the second supplementary law of quadratic reciprocity states that $( \frac{2}{p})_2$ depends only on $p\bmod 2^3$.
More precisely we have:
\begin{thm*}
For odd primes $p$, the following are equivalent.
\begin{enumerate}
\item $( \frac{2}{p} )_2=1$.
\item $p \bmod 2^3 \in \{\pm 1 \bmod 2^3\}$.
\end{enumerate}
\end{thm*}
As a generalization, it is natural to ask when $( \frac{q}{p} )_q=1$ for primes $q\geq 3$.
In this paper we provide an answer to this question under the assumption that the given prime $p$ is of the form
\begin{align*}
p=\sum_{i=0}^{q-1} m^in^{q-1-i}=\Phi_q(n/m)m^{q-1}, \ m,n \in \ZZ, \ p\neq q,
\end{align*}
where $\Phi_q(x)=\sum_{i=0}^{q-1} x^i$ denotes the $q$-th cyclotomic polynomial.
Under this assumption, an equivalent condition of  $( \frac{q}{p} )_q=1$ is given in terms of the polylogarithm of negative index
\begin{align*}
\mathrm{Li}_{1-q}(x):=\left(x\frac{d}{dx}\right)^{q-1} \frac{x}{1-x} \in \QQ(x).
\end{align*}
The following is the main result in this paper.
We denote by $\ZZ_{(q)}$ the localization of $\ZZ$ at the prime ideal $q\ZZ$.

\begin{thm} \label{prsvspl}
Let $m,n \in \ZZ$, and assume that $p:=\Phi_q(n/m)m^{q-1}$ is a prime other than $q$. 
The following are equivalent.
\begin{enumerate}
\item $( \frac{q}{p} )_q=1$.
\item $\mathrm{Li}_{1-q}(\frac{n}{m}) \in q^2\ZZ_{(q)}$.
\end{enumerate}
\end{thm}

\begin{rmk}
The assumption ``$p=\Phi_q(n/m)m^{q-1}$ is a prime other than $q$'' implies 
\begin{align} \label{equiv1}
m-n \not\equiv 0 \mod q \quad \text{and} \quad p \equiv 1 \mod q,
\end{align}
which follows from the congruence 
\begin{align*}
p=\sum_{k=0}^{q-1} m^{q-1-k} n^k\equiv \sum_{k=0}^{q-1} (-1)^k\binom{q-1}{k} m^{q-1-k} n^k=(m-n)^{q-1} \mod q.
\end{align*}
In particular $1 \neq \frac{n}{m} \in \QQ$.
On the other hand, the polylogarithm $\mathrm{Li}_{1-q}(x)$ takes the form of a polynomial of degree $q-1$ divided by $(1-x)^q$.
Hence $\mathrm{Li}_{1-q}(\frac{n}{m}) \in \QQ$ under the same assumption.
\end{rmk}

\cref{prsvspl} immediately follows by combining Theorems \ref{mainthm} and \ref{main2} below. 

\begin{dfn} \label{dfn}
We define a polynomial $f_q(x)$ of degree $q$ and a subset $S_q \subset \ZZ/q^2\ZZ$ by
\begin{align*}
&f_q(x):=
\sum_{t=1}^{q} \left(\sum_{\substack{1\leq j,k,l \leq q-1 \\ k \equiv jl \mod q}} (-1)^k\binom{t}{k}jl \right) \frac{x^t}{t}  \in \ZZ_{(q)}[x], \\
&S_q:=\left\{a-f_q(a)\bmod q^2 \in \ZZ/q^2\ZZ \mid a=0,1,\dots,q-1\right\}.
\end{align*}
\end{dfn}

\begin{thm} \label{mainthm}
Let $m,n \in \ZZ$, and assume that $p:=\Phi_q(n/m)m^{q-1}$ is a prime other than $q$. 
The following are equivalent.
\begin{enumerate}
\item $( \frac{q}{p} )_q=1$.
\item $f_q(\frac{-n}{m-n}) \in q^2\ZZ_{(q)}$.
\item $\frac{-n}{m-n} \bmod q^2 \in S_q$.
\end{enumerate}
\end{thm}

A proof of \cref{mainthm} is given in Sections \ref{ke} through \ref{calc}.
The polynomial $f_q(x)$ has a connection to the polylogarithm $\mathrm{Li}_{1-q}(x)$ as follows.

\begin{thm} \label{main2}
\begin{enumerate}
\item We have
\begin{align*}
f_q(x) \equiv c_q \mathrm{Li}_{1-q}\left(\frac{-x}{1-x}\right) \mod q^2,
\end{align*}
where $c_q$ is an integer prime to $q$.
\item Let $m,n \in \ZZ$, and assume that $m-n \not \equiv 0 \mod q$.
The following are equivalent.
\begin{enumerate}
\item $\mathrm{Li}_{1-q}(\frac{n}{m}) \in q^2\ZZ_{(q)}$.
\item $f_q(\frac{-n}{m-n}) \in q^2\ZZ_{(q)}$.
\end{enumerate}
\end{enumerate}
\end{thm}

The statements of \cref{main2} are a part of \cref{pl}, which is stated and proved in \cref{polylog}.

{\setlength{\leftmarginii}{0pt}\begin{exm} \label{exm}
For small $q$, we write down 
\begin{itemize}
\item $\overline f_q(x):=f_q(x) \bmod q^2$, 
\item $\overline{\mathrm{Li}}_{1-q}(x):=\mathrm{Li}_{1-q}(x) \bmod q^2$, 
\item $S_q$, 
\item $T_q:=\{\frac{n}{m} \bmod q^2 \mid \mathrm{Li}_{1-q}(\frac{n}{m}) \in q^2\ZZ_{(q)}\}$, putting $\frac{n}{m} \bmod q^2:=\infty$ if $\frac{m}{n} \in q^2\ZZ_{(q)}$, 
\item list of small primes of the form $p=\Phi_q(n/m)m^{q-1}$ with $( \frac{q}{p} )_q=1$, in ascending order.
\end{itemize}

\begin{enumerate}
\item $q=3$.
\begin{itemize}
\item $\overline f_3(x) =8x^{3} + 6x^{2} + 4x$, 
\item $\overline{\mathrm{Li}}_{1-3}(x)=(x^2 + x)/(1-x)^3$,
\item $S_3=\{0, 1, 5\}$,
\item $T_3=\{0,8,\infty\}$,
\item $p=61, 67, 73, 103, 151, 193, 271, 307, 367, 439, 499, 523, 547, 577, 613, 619, 643, \dots$.
\end{itemize}

\item $q=5$.
\begin{itemize}
\item $\overline f_5(x) =4x^{5} + 15x^{4} + 10x^{2} + 21x$,
\item $\overline{\mathrm{Li}}_{1-5}(x)=(x^4 + 11x^3 + 11x^2 + x)/(1-x)^5$,
\item $S_5=\{0, 1, 2, 13, 24\}$,
\item $T_5=\{0,2,13,24,\infty\}$,
\item $p=31, 19141, 30941, 48871, 114641, 125591, 141961, 170101, 225241, 246931, \dots$
\end{itemize}

\item $q=7$.
\begin{itemize}
\item $\overline f_7(x) =20x^{7} + 28x^{6} + 28x^{5} + 7x^{4} + 14x^{3} + 35x^{2} + 15x$, 
\item $\overline{\mathrm{Li}}_{1-7}(x)=(x^6 + 8x^5 + 8x^4 + 8x^3 + 8x^2 + x)/(1-x)^7$,
\item $S_7=\{0, 1, 6, 17, 25, 33, 44\}$,
\item $T_7=\{0,9,11,24,47,48,\infty\}$,
\item $p=43, 10501, 3692053, 109894303, 115928821, 138520537, 141903217,  \dots$.
\end{itemize}
\end{enumerate}
Some general properties of $S_q$, $T_q$, including that $0,\frac{1}{2},1\bmod q^2 \in S_q$ and $0,-1,\infty \in T_q$, are given in \cref{pl}-(3).
\end{exm}}

\begin{rmk} \label{ec}
\begin{enumerate}
\item Euler conjectured several formulas for cubic residue.
Let $p$ be a prime with $p \equiv1 \bmod 3$. Then we can write $4p=L^2+27M^2$ with $L,M\in \ZZ$.
In this setting, one of Euler's conjectures (\cite[Proposition 7.2]{Lem}) states that the following are equivalent.
\begin{enumerate}
\item $(\frac{3}{p})_3=1$.
\item $M\equiv 0 \mod 3$.
\end{enumerate}
This statement is equivalent to ``(1)$\LR$(3)-part'' of \cref{mainthm} with $q=3$, in view of \cref{exm}-(1).
The correspondence of notations is 
\begin{align*}
(L,M)=
\begin{cases}
(m+2n,\frac{m}{3}) & (3 \mid m), \\
(2m+n,\frac{n}{3}) & (3 \mid n), \\
(m-n,\frac{m+n}{3}) & (3 \mid (m+n)).
\end{cases}
\end{align*}
Hence we see that $M \equiv 0 \mod 3 \LR$ $m$, $n$ or $m+n \equiv 0 \mod 9 \LR \frac{-n}{m-n} \equiv 0,1,5 \mod 9$.
\item The case of $q=3$ is related also to the Feit-Thompson conjecture (\cite{FT}).
Let $p,q$ be any distinct primes. Then the Feit-Thompson conjecture states that 
\begin{align*}
\Phi_q(p) \nmid \Phi_p(q).
\end{align*}
By a similar argument to that in \cite[Proof of Theorem]{Le}, we see that \cref{mainthm} with $q=3$ implies the Feit-Thompson conjecture with $q=3$ and $p \not \equiv 8 \bmod 9$.
We note that although \cite[Theorem]{Le} states a stronger statement, there seems to be a mistake in \cite[Lemma 2, needing to drop ``coprime'' in the statement]{Le}. 
\item The fourth author, in his Master's thesis (in Japanese), conjectured a weaker version of \cref{mainthm} and proved it in the case of $q=5$.
He used the following analogue of Euler's conjecture, which was proved in e.g. \cite[p122]{Leh}, \cite[Theorem 8]{Di}, \cite[(7)]{Mu}:
let $p$ be a prime with $p\equiv 1 \bmod 5$. Then we can write with $x,u,v,w \in \ZZ$
\begin{align*}
16p&=x^2+50u^2+50v^2+125w^2, \\
xw&=v^2-4uv-u^2, \\
x&\equiv 1 \mod 5.
\end{align*}
Then the following are equivalent.
\begin{enumerate}
\item $(\frac{5}{p})_5=1$
\item $u\equiv 2v \mod 5$.
\end{enumerate}
\end{enumerate}
\end{rmk}

The outline of this paper is as follows. 
We first prove \cref{mainthm} in Sections \ref{ke} through \ref{calc}.
In \cref{ke}, we relate the question to a suitable Kummer extension over $\QQ(\zeta_q)$ with $\zeta_q$ a primitive $q$-th root of unity.
This argument is analogous to Lemmermeyer's proof of Euler's conjectures for cubic residue (\cite[\S 7]{Lem}).
We give an explicit element $\mu^{(i)}\in \ZZ[\zeta_q]$ whose $q$-th root is a generator of this Kummer extension (\cref{im}).
Then, in (\ref{lc}), we see that $( \frac{q}{p} )_q=1$ is equivalent to $\mu^{(i)}$ being a $q$-th power element $\bmod$ $(1-\zeta_q)^{q+1}$. 
\cref{orders} is a key result although its proof is postponed due to the lengthy calculations required.
In \cref{profmain}, we give complete representatives of $\ZZ[\zeta_q]/(1-\zeta_q)^e$ ($e=q$, $q+1$) and its $q$-th power elements and prove \cref{mainthm}.
In \cref{calc}, we calculate the value of the $q$-adic $\log$ of $\mu^{(i)}$ modulo $(1-\zeta_q)^e$ and prove \cref{orders}, which is postponed in \cref{ke}.
In \cref{polylog}, we provide an explicit relation between $f_q(x)$ and the polylogarithm $\mathrm{Li}_{1-q}(x)$, which derives the symmetric properties of $f_q(x)$ and $S_q$.
The main result in this section is \cref{pl}, which implies \cref{main2}.

\subsection*{Acknowledgement}

This work was supported by JSPS KAKENHI Grant Numbers JP21K03185, JP21K13779, JP22K03253 and JST, CREST Grant Number JPMJCR1913, Japan.

\subsection*{Notation}
We denote the cardinality of a set $S$ by $|S|$. 
For $k \in \NN$, we put $\zeta_k:=\exp(2 \pi \sqrt{-1}/k)$ and $\Phi_k$ to be the minimal polynomial of $\zeta_k$ over $\QQ$.
For a prime $q$, 
we normalize the $q$-adic additive valuation, which we denote by $v_q$, so that $v_q(q)=1$.
We denote the localization of $\ZZ$ at the prime ideal $q\ZZ$ by $\ZZ_{(q)}$.
For $a \in \ZZ_{(q)}$ with $q \nmid a$, we put $a^*$ to be the integer satisfying 
\begin{align*}
1\leq a^* \leq q-1, \quad aa^*\equiv 1 \mod q. 
\end{align*}

\section{Kummer extension related to $(\frac{q}{p})_q$} \label{ke}

In this section, we see that $(\frac{q}{p})_q=1$ is equivalent to a congruence condition on a generator $\mu$ of a suitable Kummer extension in the last line in (\ref{lc}),
and give an explicit generator in \cref{im}.
Let $p$ be an odd prime satisfying $p\equiv 1 \mod q$. We consider the following diagram.
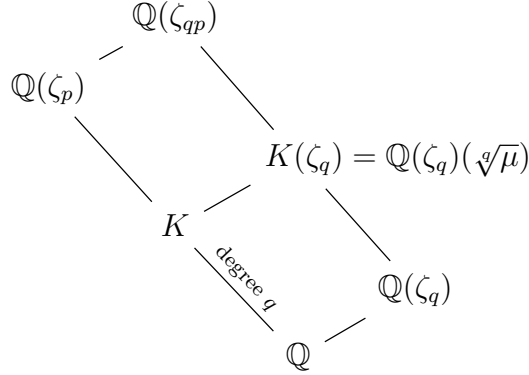
\begin{center}
\begin{tikzcd}[row sep=tiny, column sep=tiny]
&\QQ(\zeta_{qp}) \ar[dl, dash] \ar[ddr, dash]\\
\QQ(\zeta_{p}) \ar[ddr, dash]\\
&&K(\zeta_q)=\QQ(\zeta_q)(\sqrt[q]{\mu}) \hspace*{-75pt} \ar[dl, dash] \ar[ddr, dash]\\
&K \ar[ddr, dash, "\text{degree }q" sloped]\\
&&&\QQ(\zeta_q) \ar[dl, dash]\\
&&\QQ
\end{tikzcd}
\end{center}
Here we put $K$ to be the unique subfield of $\QQ(\zeta_p)$ with $[K:\QQ]=q$. 
The field $K(\zeta_q)$ is characterized by the following: 
\begin{equation} \label{char} 
\begin{split}
&\text{$K(\zeta_q)$ is an abelian extension over $\QQ$ of degree $q(q-1)$} \\
&\text{and $K(\zeta_q)$ is an unramified extension over $\QQ(\zeta_q)$ outside of $p$.} 
\end{split}
\end{equation}
The uniqueness is due to the Kronecker-Weber Theorem, which implies $K(\zeta_q) \subset \QQ(\zeta_{qp^\infty})$.
We see the following equivalences, which is a complete analogue of arguments in \cite[\S 7.1]{Lem}.
In particular, the last equivalence follows from the decomposition law \cite[Theorem 4.12- c)]{Lem} for Kummer extensions.
\begin{equation} \label{lc}
\begin{split}
( \tfrac{q}{p})_q=1 \quad &\LR \quad q^\frac{p-1}{q} \equiv 1 \mod p \\
&\LR \quad \text{$q$ splits in $K/\QQ$} \\
&\LR \quad \text{$(1-\zeta_q)$ splits in $K(\zeta_q)/\QQ(\zeta_q)$} \\
&\LR \quad \text{$\mu \equiv \xi^q \mod (1-\zeta_q)^{q+1}$ for some $\xi \in \ZZ[\zeta_q]$},
\end{split}
\end{equation}
where $\mu \in \ZZ[\zeta_q]$ is any element satisfying $K(\zeta_q)=\QQ(\zeta_q)(\sqrt[q]{\mu})$ and $(1-\zeta_q) \nmid \mu$.

Hereinafter, we return to the setting in \cref{s1}:
\begin{center}
$p=\Phi_q(n/m)m^{q-1}$ is a prime other than $q$,
\end{center}
which implies $p \equiv 1 \mod q$ by \cref{equiv1}.
In this setting, we obtain an explicit element $\mu$ applicable to (\ref{lc}) as in \cref{im} below.

\begin{dfn} \label{mu}
For $i \in \ZZ$ with $0 \leq i \leq q-1$ we put 
\begin{align*} 
\mu^{(i)}=\zeta_q^i \prod_{j=1}^{q-1} (m-n\zeta_q^j)^{j^*} \in \ZZ[\zeta_q].
\end{align*} 
\end{dfn}

The following Proposition follows from Kummer theory.

\begin{prp} \label{ur}
Assume that $0 \leq i \leq q-1$.
\begin{enumerate}
\item $\QQ(\zeta_q)(\sqrt[q]{\mu^{(i)}})/\QQ$ is an abelian extension.
\item $\QQ(\zeta_q)(\sqrt[q]{\mu^{(i)}})/\QQ(\zeta_q)$ is an unramified extension outside of $qp$.
\item  $(1-\zeta_q)$ is unramified in $\QQ(\zeta_q)(\sqrt[q]{\mu^{(i)}})/\QQ(\zeta_q)$ if and only if $\mu^{(i)} \equiv \xi^q \mod (1-\zeta_q)^q$ for some $\xi \in \ZZ[\zeta_q]$.
\end{enumerate}
\end{prp}

\begin{proof}
(1) follows from \cite[Corollary 4.17]{Lem}.

\medskip

\noindent
(2) and (3) follow from \cite[Theorem 4.12]{Lem} or \cite[Theorem 6.3]{Gr}.
Here we note that $(m-n\zeta_q^j)$ ($1\leq j \leq q-1$) are all prime ideals lying above $p=\prod_{j=1}^{q-1} (m-n\zeta_q^j)$
and hence $(1-\zeta_q) \nmid \mu^{(i)}$.
\end{proof}

Our proof of the following Theorem is rather lengthy, so it is postponed to \cref{calc}.

\begin{thm} \label{orders}
We take $0 \leq i \leq q-1$ so that  
\begin{align*}
i\equiv \frac{-n}{m-n} \mod q.
\end{align*}
\begin{enumerate}
\item $\mu^{(i)} \equiv (m-n)^{\frac{(q-1)q}{2}} \mod (1-\zeta_q)^q$.
\item The following are equivalent.
\begin{enumerate}
\item $\mu^{(i)} \equiv (m-n)^{\frac{(q-1)q}{2}} \mod (1-\zeta_q)^{q+1}$.
\item $f_q(\frac{-n}{m-n}) \in q^2\ZZ_{(q)}$.
\item $\frac{-n}{m-n}\bmod q^2 \in S_q$.
\end{enumerate}
\end{enumerate}
\end{thm}

\begin{crl} \label{im}
Let $m,n,K$ be as above. That is, $p=\Phi_q(n/m)m^{q-1}$ is a prime other than $q$ and $K$ is the unique subfield of $\QQ(\zeta_{p})$ with $[K:\QQ]=q$. 
Assume that $0 \leq i \leq q-1$ and $i\equiv \frac{-n}{m-n} \mod q$.
Then a $q$-th root $\sqrt[q]{\mu^{(i)}}$ is a generator of the Kummer extension $K(\zeta_q)/\QQ(\zeta_q)$ satisfying $(1-\zeta_q) \nmid \mu^{(i)}$. 
\end{crl}

\begin{proof}
Since $K(\zeta_q)$ is characterized by (\ref{char}), it follows from \cref{ur} and \cref{orders}-(1).
\end{proof}

\section{$q$-th power elements modulo a power of $(1-\zeta_q)$} \label{profmain}

In this section, we prove \cref{mainthm}  by studying a condition for an element in $\ZZ[\zeta_q]$ being a $q$-th power element modulo a power of $(1-\zeta_q)$.
We prepare the following Proposition.

\begin{prp} \label{unram}
\begin{enumerate}
\item A complete representative set of $\ZZ[\zeta_q]/(1-\zeta_q)^q$ is given by
\begin{align} \label{rs}
\left\{a_0+a_1(1-\zeta_q)+\dots+a_{q-2}(1-\zeta_q)^{q-2} \ \bigg | \ \genfrac{}{}{0pt}{}{\displaystyle a_i \in \ZZ, \ 0\leq a_0 <q^2}{\displaystyle 0\leq a_1,\dots,a_{q-2} <q} \right\}.
\end{align}
The representatives of $\{\xi^q \bmod (1-\zeta_q)^q \mid \xi \in \ZZ[\zeta_q]\}$ in (\ref{rs}) are given by
\begin{align*}
\left\{a_0+a_1(1-\zeta_q)+\dots+a_{q-2}(1-\zeta_q)^{q-2} \ \bigg | \ \genfrac{}{}{0pt}{}{\displaystyle a_0 \equiv b^q \bmod q^2 \text{ for some } b \in \ZZ}{\displaystyle a_1,\dots,a_{q-2}=0} \right\}.
\end{align*}
\item A complete representative set of $\ZZ[\zeta_q]/(1-\zeta_q)^{q+1}$ is given by
\begin{align} \label{rs2}
\left\{a_0+a_1(1-\zeta_q)+\dots+a_{q-2}(1-\zeta_q)^{q-2} \ \bigg | \ \genfrac{}{}{0pt}{}{\displaystyle a_i \in \ZZ, \ 0\leq a_0,a_1 <q^2}{\displaystyle 0\leq a_2,\dots,a_{q-2} <q} \right\}.
\end{align}
The representatives of $\{\xi^q \bmod (1-\zeta_q)^{q+1} \mid \xi \in \ZZ[\zeta_q]\}$ in (\ref{rs2}) are given by
\begin{multline*} 
\biggl \{a_0+a_1(1-\zeta_q)+\dots+a_{q-2}(1-\zeta_q)^{q-2} \\
\bigg | \ 
\genfrac{}{}{0pt}{}{[\displaystyle \,a_0=0,\ a_1 \equiv 0 \bmod q,\ a_2,\dots,a_{q-2}=0\,]\text{ or }}{[\displaystyle \,a_0 \equiv b^q \bmod q^2 \text{ for some } b \in \ZZ,\ a_1,a_2,\dots,a_{q-2}=0\,]}
\biggr\}. 
\end{multline*}
\end{enumerate}
\end{prp}

\begin{proof}
(1) For $a_0,a_1,\dots,a_{q-2} \in \ZZ$, we have
\begin{equation} \label{val}
\begin{split}
&v_q(a_0+a_1(1-\zeta_q)+\dots+a_{q-2}(1-\zeta_q)^{q-2}) \\
&=\min \bigl\{ v_q(a_0),v_q(a_1(1-\zeta_q)),\dots,v_q(a_{q-2}(1-\zeta_q)^{q-2})\bigr\}.
\end{split}
\end{equation}
Hence $a_0+a_1(1-\zeta_q)+\dots+a_{q-2}(1-\zeta_q)^{q-2} \in (1-\zeta_q)^q$ is equivalent to $a_0 \in q^2\ZZ$ and $a_1,\dots, a_{q-2} \in q\ZZ$.
Therefore two different elements in (\ref{rs}) are not congruent modulo $(1-\zeta_q)^q$.
Then the first statement follows since $|\ZZ[\zeta_q]/(1-\zeta_q)^q|=q^q=$ the number of elements in (\ref{rs}).
The second statement follows from the congruence $(b_0+b_1(1-\zeta_q)+\dots+b_{q-2}(1-\zeta_q)^{q-2})^q \equiv b_0^q \mod (1-\zeta_q)^q$.

\medskip

\noindent 
(2) The first statement follows by a similar argument to that of (1). 
We prove the second statement.
The minimal polynomial of $1-\zeta_q$ is of the form 
\begin{align*}
x^{q-1}+x^{q-2}+\dots+1 |_{x=1-y} = y^{q-1}+qc_{q-2} y^{q-2} +\dots+qc_1y+q \ (c_i \in \ZZ).
\end{align*}
Then we see that 
\begin{align*}
(1-\zeta_q)^q&=(-qc_{q-2} (1-\zeta_q)^{q-2} -\dots -qc_1(1-\zeta_q)-q)(1-\zeta_q) \\
&\equiv -q(1-\zeta_q) \mod (1-\zeta_q)^{q+1}.
\end{align*}
Hence we obtain 
\begin{align*}
&(b_0+b_1(1-\zeta_q)+\dots+b_{q-2}(1-\zeta_q)^{q-2})^q \\
& \equiv b_0^q+ qb_0^{q-1}b_1(1-\zeta_q)+b_1^q(1-\zeta_q)^q \\
& \equiv b_0^q+ qb_1(b_0^{q-1}-b_1^{q-1})(1-\zeta_q) \\
& \equiv
\begin{cases}
 (- q b_1^q)(1-\zeta_q) \mod (1-\zeta_q)^{q+1}& (q \mid b_0) \\
b_0^q \mod (1-\zeta_q)^{q+1}& (q \nmid b_0)
\end{cases}
\end{align*}
as desired.
\end{proof}

Based on the above, we can rephrase the statement of \cref{orders} as follows.

\begin{dfn} \label{mua}
Let $\mu^{(i)}=\zeta_q^i \prod_{j=1}^{q-1} (m-n\zeta_q^j)^{j^*} \in \ZZ[\zeta_q]$ be as in \cref{mu}.
We define $a_k^{(i)} \in \ZZ$ for $0 \leq i \leq q-1$, $0\leq k \leq q-2$ by
\begin{align*}
\mu^{(i)} =\sum_{k=0}^{q-2} a_k^{(i)} (1-\zeta_q)^k.
\end{align*}
\end{dfn}

\begin{thm} \label{1}
Assume that $0 \leq i \leq q-1$ and $i\equiv \frac{-n}{m-n} \mod q$. 
\begin{enumerate}
\item $a_0^{(i)} \equiv (m-n)^{\frac{(q-1)q}{2}} \mod q^2$ and $a_1^{(i)} \equiv \cdots \equiv a_{q-2}^{(i)} \equiv 0 \mod q$.
\item The following are equivalent.
\begin{enumerate}
\item $a_1^{(i)} \equiv 0 \mod q^2$.
\item $f_q(\frac{-n}{m-n}) \in q^2\ZZ_{(q)}$.
\item $\frac{-n}{m-n}  \bmod q^2 \in S_q$.
\end{enumerate} 
\end{enumerate}
\end{thm}

\begin{proof}
(1) We note that 
\begin{align*}
\mu^{(i)}-(m-n)^{\frac{(q-1)q}{2}}=(a_0^{(i)}-(m-n)^{\frac{(q-1)q}{2}})+\sum_{k=1}^{q-2} a_k^{(i)}(1-\zeta_q)^k.
\end{align*}
Then the assertion follows from \cref{orders}-(1) and \cref{val}.

\medskip

\noindent
(2) By (1) we have
\begin{align*}
\mu^{(i)}-(m-n)^{\frac{(q-1)q}{2}}\equiv a_1^{(i)}(1-\zeta_q) \mod (1-\zeta_q)^{q+1}.
\end{align*}
Then the assertion follows from \cref{orders}-(2).
\end{proof}

\begin{proof}[\bf Proof of \cref{mainthm}]
Take $i$ so that $i\equiv \frac{-n}{m-n} \mod q$. Then $\sqrt[q]{\mu^{(i)}}$ is a generator of the Kummer extension $K(\zeta_q)/\QQ(\zeta_q)$ with $(1-\zeta_q) \nmid \mu^{(i)}$ by \cref{im}.
In this setting (\ref{lc}) states that 
\begin{center}
$( \tfrac{q}{p})_q=1$ if and only if $\mu^{(i)} \equiv \xi^q \mod (1-\zeta_q)^{q+1}$ for some $\xi \in \ZZ[\zeta_q]$.
\end{center}
By \cref{unram}-(2), it is further equivalent to 
\begin{align*}
a_0^{(i)} \equiv b^q \mod q^2 \text{ for some } b \in \ZZ, \ a_1^{(i)} \equiv 0 \mod q^2,\ a_2^{(i)},\dots,a_{q-2}^{(i)}\equiv 0 \mod q.
\end{align*}
By \cref{1}-(1), the assumption $i\equiv \frac{-n}{m-n} \mod q$ implies the conditions other than $ a_1^{(i)} \equiv 0 \mod q^2$.
Then \cref{mainthm} follows from \cref{1}-(2).
\end{proof}

\section{Calculation of $\mu^{(i)}$ modulo a power of $(1-\zeta_q)$} \label{calc}

In this section, we calculate $\mu^{(i)} \bmod (1-\zeta_q)^{q+1}$ and prove \cref{orders}.
To begin with, we study $\mu^{(i)} \bmod (1-\zeta_q)^2$.

\begin{prp} \label{muru}
Let $\mu^{(i)}$ be as in \cref{mu}.
\begin{enumerate}
\item $(m-n)^\frac{(q-1)q}{2} \equiv (\frac{m-n}{q})_2  \mod q^2$.
\item We have 
\begin{align*}
\mu^{(i)} \equiv (m-n)^\frac{(q-1)q}{2} \zeta_q^{i+n(m-n)^*} \mod (1-\zeta_q)^2.
\end{align*}
Moreover the following are equivalent.
\begin{enumerate}
\item $i \equiv \frac{-n}{m-n} \mod q$.
\item $\mu^{(i)} \equiv (m-n)^\frac{(q-1)q}{2} \mod (1-\zeta_q)^2$.
\end{enumerate}
\end{enumerate}
\end{prp}

\begin{proof}
(1) We note that $q \nmid (m-n)$ (\cref{equiv1}). 
By Euler's criterion for quadratic residue, we can write $(m-n)^{\frac{q-1}{2}}=(\frac{m-n}{q})_2 +qk$ with $k \in \ZZ$.
Then we have $(m-n)^{\frac{(q-1)q}{2}}=(\frac{m-n}{q})_2^q +q(\frac{m-n}{q})_2^{q-1}qk+\cdots \equiv  (\frac{m-n}{q})_2 \mod q^2$.

\medskip

\noindent
(2) We have
\begin{align*}
1-\zeta_q^{a}=(1-\zeta_q)(1+\zeta_q+\dots+\zeta_q^{(a-1)}) \equiv a(1-\zeta_q) \mod (1-\zeta_q)^2
\end{align*}
for $a \in \NN$.
It immediately follows that $1-\zeta_q^{ac^*b} \equiv \frac{a}{c}(1-\zeta_q^b) \mod (1-\zeta_q)^2$ for $a,b,c\in \ZZ$ with $q \nmid c$.
Hence we obtain the first congruence as follows:
\begin{equation} \label{tcl} 
\begin{split}
\mu^{(i)}
&=(m-n)^{\frac{(q-1)q}{2}}\cdot \zeta_q^i \prod_{j=1}^{q-1} (1+n(m-n)^{-1}(1-\zeta_q^j))^{j^*}  \\
&\equiv (m-n)^{\frac{(q-1)q}{2}}\cdot \zeta_q^i \prod_{j=1}^{q-1} (1-(1-\zeta_q^{-n(m-n)^*j}))^{j^*}  \mod  (1-\zeta_q)^2 \\
&=(m-n)^{\frac{(q-1)q}{2}}\cdot \zeta_q^{i+n(m-n)^*}. 
\end{split}
\end{equation}
For any primitive $k$-th root of unity $\zeta \neq 1$, we see that

\begin{align} \label{v(1-z)}
v_q(1-\zeta)=
\begin{cases}
1/(q^l-q^{l-1}) & (k=q^l,\ l\in \NN) \\
0 & (\text{otherwise}) 
\end{cases}
< 2/(q-1)=v_q((1-\zeta_q)^2).
\end{align}
Thus the equivalence (a) $\LR$ (b) follows.
\end{proof}

In order to prove \cref{orders}, we would like to refine \cref{muru}, from modulo $ (1-\zeta_q)^2$ to modulo $ (1-\zeta_q)^{e}$ ($e=q$, $q+1$).
For this purpose, the following Lemma is very useful since $\mu^{(i)}$ is defined as a product.
We denote the $q$-adic $\log$ function and the $q$-adic $\exp$ function by $\log_q$ and $\exp_q$ respectively (see, for example, \cite[Chap.~IV, \S 1]{Ko} for more details).
We denote by $\QQ_q$, $\overline{\QQ_q}$ and $\CC_q$, the field of $q$-adic numbers, the algebraic closure of $\QQ_q$ and the completion of $\overline{\QQ_q}$, respectively.

\begin{lmm} \label{nt}
Let $A \in \QQ(\zeta_q)$. 
Assume that 
\begin{align*}
A \equiv \zeta \mod (1-\zeta_q)^2
\end{align*}
for a root of unity $\zeta \in \overline{\QQ_q}$.
For $e \in \NN$ with $e\geq 2$, the following are equivalent.
\begin{enumerate}
\item $A \equiv \zeta \mod (1-\zeta_q)^e$.
\item $\log_q A \equiv 0 \mod (1-\zeta_q)^e$.
\end{enumerate}
\end{lmm}

\begin{proof}
By properties of $\log_q$ and $\exp_q$, we have
\begin{enumerate}
\item[(a)] $\log_q(\{1+z \in \CC_q \mid v_{q}(z) \geq v\})=\{z \in \CC_q\mid v_{q}(z) \geq v\}$  ($v> 1/(q-1)$).
\item[(b)] $\exp_q(\{z \in \CC_q\mid v_{q}(z) \geq v\})=\{1+z \in \CC_q \mid v_{q}(z) \geq v\}$  ($v> 1/(q-1)$).
\item[(c)] $z/\exp_q(\log_q(z)) \in \{\alpha \in \CC_q^\times \mid \text{there exist $k\in \NN$, $l \in \ZZ$ satisfying }\alpha^k=q^l \}$ ($z \in \CC_q^\times$).
\end{enumerate}
By (a), we see that (1) implies $\log_q A \equiv \log_q(\zeta) \equiv 0 \mod (1-\zeta_q)^e$.
We prove the converse.
We note that the assumption $A \equiv \zeta \mod (1-\zeta_q)^2$ implies that $A$ is a $q$-unit.
Then, by (b), (c) and (2), there exists a root of unity $\zeta'$ satisfying
\begin{align*}
A\zeta'=\exp_q(\log_q(A)) \equiv 1 \mod (1-\zeta_q)^e,
\end{align*}
which implies $\zeta \zeta' \equiv 1 \bmod (1-\zeta_q)^2$.
Hence, by \cref{v(1-z)}, we have $\zeta \zeta'=1$ and $A \equiv \zeta'^{-1}\equiv \zeta \mod (1-\zeta_q)^e$ as desired. 
\end{proof}

We recall some formulas for binomial coefficients, which we use in order to prove \cref{orders}. 

\begin{prp} \label{bc}
\begin{enumerate}
\item $\binom{t}{k}=0$ ($k>t$).
\item $\sum_{k=0}^t (-1)^k\binom{t}{k}=
\begin{cases}
0 & (t\geq 1) \\
1 & (t=0)
\end{cases}$.
\item For $a \in \ZZ_{(q)}$ we have
\begin{align*}
\sum_{\substack{1\leq t \leq q \\ 1\leq k \leq q-1}} (-1)^{k}\binom{t-1}{k-1}a ^t \equiv 0 \mod q.
\end{align*}
\end{enumerate}
\end{prp}

\begin{proof}
(1), (2) are well-known.

\medskip

\noindent
(3) Dividing $\sum_{\substack{1\leq t \leq q \\ 1\leq k \leq q-1}} $ into 
$\sum_{\substack{ t=1 \\ 1\leq k \leq q-1}}+\sum_{\substack{2\leq t \leq q-1 \\ 1\leq k \leq q-1}} +\sum_{\substack{t =q \\ 1\leq k \leq q-1}}$ and applying (1) and (2) we have
\begin{align*}
\sum_{\substack{1\leq t \leq q \\ 1\leq k \leq q-1}} (-1)^{k}\binom{t-1}{k-1}a ^t =-a +0+ a ^q \equiv 0 \mod q
\end{align*}
as desired.
\end{proof}

The following properties of the polynomial $f_q(x)$ are key ingredients of our proof of \cref{orders}.

\begin{prp} \label{fqgq} 
Let $f_q(x)$, $S_q$ be as in \cref{dfn} and $\mu^{(i)}$ as in \cref{mu}.
\begin{enumerate}
\item $f_q(x) \equiv x-x^q \mod q$. In particular we have $f_q(c) \equiv 0 \mod q$ for every $c\in \ZZ_{(q)}$.
\item $S_q=\{b \in \ZZ/q^2\ZZ \mid f_q(b)\equiv 0 \bmod q^2\}$.
\item $\log_q (\mu^{(i)})\equiv (1-\zeta_q) f_q(\frac{-n}{m-n}) \mod (1-\zeta_q)^{q+1}$.
In particular we have the following statements.
\begin{enumerate}
\item $\log_q (\mu^{(i)})\equiv 0 \mod (1-\zeta_q)^q$.
\item $\log_q (\mu^{(i)})\equiv 0 \mod (1-\zeta_q)^{q+1}$ if and only if $f_q(\tfrac{-n}{m-n}) \in q^2\ZZ_{(q)}$.
\end{enumerate}
\end{enumerate}
\end{prp}

\begin{proof}
(1) By \cref{bc}-(2) we see that
\begin{align*}
\sum_{\substack{1\leq j,k,l \leq q-1 \\ k \equiv jl \mod q} (-1)^k\binom{t}{k}\frac{jl}{t}} 
\equiv -\sum_{1\leq k \leq q-1} (-1)^k\binom{t-1}{k-1} 
=
\begin{cases}
1 & (t=1) \\
-1 & (t=q) \\
0 & (1<t<q)
\end{cases}
\mod q.
\end{align*}
Then the assertion follows.

\medskip

\noindent
(2) For $0\leq a \leq q-1$, we define $b(a)\in \ZZ_{(q)}$ and $\alpha_d(a) \in \ZZ_{(q)}$ ($0\leq d \leq q$) by
\begin{align*}
b(a)=a-f_q(a), \quad f_q(x)=\sum_{d=0}^{q} \alpha_d(a) (x-a)^d.
\end{align*}
By (1) we have $\alpha_0(a)=f_q(a) \equiv 0 \mod q$.
Since (1) implies $f_q'(x) \equiv 1 \mod q$, we also have $\alpha_1(a)=f_q'(a) \equiv 1 \mod q$.
It follows that 
\begin{align*}
f_q(b(a))=\sum_{d=0}^{q} \alpha_d (a)(-f_q(a))^d \equiv 0 \mod q^2.
\end{align*}
Thus we have
\begin{align*}
S_q=\{b(a) \ \bmod \ q^2 \mid 0\leq a \leq q-1\} \subset \{b \in \ZZ/q^2\ZZ \mid f_q(b)\equiv 0 \bmod q^2\}.
\end{align*}
The last set consists of exactly $q$ elements by (1) and Hensel's lemma.
We also have $|S_q|=q$ since $b(0),b(1),\dots b(q-1)$ are pairwise distinct modulo $q$ by (1), 
and hence pairwise distinct modulo $q^2$.
Then the assertion holds.

\medskip

\noindent
(3) Let $\alpha:=-n(m-n)^{-1} \in \ZZ_{(q)}$. Taking $\log_q$ of the first line of \cref{tcl}, we have
\begin{align*}
\log_q (\mu^{(i)})\equiv \sum_{j=1}^{q-1} j^*\log_q (1-\alpha(1-\zeta_q^j)) \mod q^2
\end{align*}
since \cref{muru}-(1) implies $\log_q((m-n)^{\frac{(q-1)q}{2}})\equiv \log_q(\pm 1)=0 \mod q^2$.
By the definition of $\log_q$ as the usual power series and the binomial expansion, this becomes
\begin{equation} \label{mu1}
\begin{split}
\log_q (\mu^{(i)})&\equiv \sum_{j=1}^{q-1} \sum_{t=1}^q \frac{-j^*(\alpha(1-\zeta_q^{j}))^t}{t} \mod (1-\zeta_q)^{q+1} \\
&=\sum_{\substack{1\leq j \leq q-1 \\ 1\leq t \leq q \\ 0\leq k \leq t}} \frac{-(-1)^kj^*\binom{t}{k}\alpha^t\zeta_q^{jk}}{t} \\
&= \sum_{l=0}^{q-1}\left(\sum_{\substack{1\leq j \leq q-1 \\ 1\leq t \leq q \\ 0\leq k \leq t \\ k \equiv jl \mod q}} \frac{-(-1)^kj\binom{t}{k}\alpha^t}{t} \right)\zeta_q^{l}. 
\end{split}
\end{equation}
In the last equality, we replace $l$ to $jk$ and $j$ to $j^*$.
On the other hand, since there exists a unique $0\leq l \leq q-1$ satisfying $k \equiv jl \mod q$ for each $j,k$, we see that  
\begin{align*}
\sum_{l=0}^{q-1} \left(\sum_{\substack{1\leq j \leq q-1 \\ 1\leq t \leq q \\ 0\leq k \leq t \\ k \equiv jl \mod q}} \frac{-(-1)^kj\binom{t}{k}\alpha^t}{t} \right) \cdot 1
=\sum_{\substack{1\leq j \leq q-1 \\ 1\leq t \leq q}} \frac{-j\alpha^t}{t} \sum_{k=0}^t (-1)^k\binom{t}{k}=0
\end{align*}
by \cref{bc}-(2).
Hence from \cref{mu1} we obtain 
\begin{align*}
\log_q (\mu^{(i)})&\equiv\sum_{l=1}^{q-1}\left(\sum_{\substack{1\leq j \leq q-1 \\ 1\leq t \leq q \\ 0\leq k \leq t \\ k \equiv jl \mod q}} \frac{(-1)^kj\binom{t}{k}\alpha^t}{t} \right) (1-\zeta_q^{l})  \mod (1-\zeta_q)^{q+1} \\
&= (1-\zeta_q)\sum_{l=1}^{q-1}\left(\sum_{\substack{1\leq t \leq q \\ 1\leq j,k \leq q-1 \\ k \equiv jl \mod q}} \frac{(-1)^kj\binom{t}{k}\alpha^t}{t} \right) (1+\dots+\zeta_q^{l-1}). 
\end{align*}
For modification of the range of $k$, we used \cref{bc}-(1) and $k\neq 0,q$ which follows from $k \equiv jl\not \equiv  0 \mod q$.
Finally, by \cref{bc}-(3), we have  
\begin{align*}
\sum_{\substack{1\leq t \leq q \\ 1\leq j,k \leq q-1 \\ k \equiv jl \mod q}} \frac{(-1)^kj\binom{t}{k}\alpha^t}{t} &=
\sum_{\substack{1\leq t \leq q \\ 1\leq j,k \leq q-1 \\ k \equiv jl \mod q}} \frac{(-1)^kj\binom{t-1}{k-1}\alpha^t}{k} 
\equiv \sum_{\substack{1\leq t \leq q \\ 1\leq k \leq q-1}} \frac{(-1)^k\binom{t-1}{k-1}\alpha^t}{l} \equiv 0 \mod q,
\end{align*}
which implies that
\begin{align*}
\log_q (\mu^{(i)})&\equiv (1-\zeta_q)\left(\sum_{\substack{1\leq t \leq q \\ 1\leq j,k,l \leq q-1 \\ k \equiv jl \mod q}} \frac{(-1)^k j \binom{t}{k}\alpha^t}{t}\right) \cdot l = (1-\zeta_q) f_q(\alpha) \mod (1-\zeta_q)^{q+1}. 
\end{align*}
as desired. Then (b) is clear and (a) follows from (1).
\end{proof}

\begin{proof}[\bf Proof of \cref{orders}]
We assume that $i\equiv \frac{-n}{m-n} \mod q$.
Then $\mu^{(i)} \equiv (m-n)^{\frac{(q-1)q}{2}} \equiv (\frac{m-n}{q})_2 \mod (1-\zeta_q)^2$ by \cref{muru}.
Hence we obtain $\mu^{(i)} \equiv (\frac{m-n}{q})_2 \mod (1-\zeta_q)^q$ by \cref{nt} and \cref{fqgq}-(3)-(a).
Thus \cref{orders}-(1) follows from \cref{muru}-(1).

For \cref{orders}-(2), 
by applying a similar argument based on \cref{muru} and \cref{nt}, we see that (a) is equivalent to $\log_q (\mu^{(i)})\equiv 0 \mod (1-\zeta_q)^{q+1}$.
This is further equivalent to (b) by \cref{fqgq}-(3)-(b).
The equivalence (b) $\LR$ (c) follows from \cref{fqgq}-(2).
\end{proof}

\section{Relation to the polylogarithm $\mathrm{Li}_{1-q}(z)$} \label{polylog}
In this section, we provide an explicit relation between $f_q(x)$ and the polylogarithm $\mathrm{Li}_{1-q}(x)$.
Then we derive certain symmetric properties of $f_q(x)$ and $S_q$ from a functional equation of $\mathrm{Li}_{1-q}(x)$.

\begin{dfn}
\begin{enumerate}
\item The polylogarithm $\mathrm{Li}_{-s}(x)$ with $0\leq s \in \ZZ$ is defined by
\begin{align*}
\mathrm{Li}_{-s}(x)=\left(x\frac{d}{dx}\right)^s \frac{x}{1-x} \in \QQ(x).
\end{align*} 
\item We define for $0\leq s \in \ZZ$ 
\begin{align*}
\mathcal F_s(x):=\sum_{t=0}^s \sum_{\substack{s_1,\dots,s_t\geq 1 \\ s_1+\dots+s_t=s}} \binom{s}{s_1,\dots,s_t}x^t \in \ZZ[x] \quad \text{for} \quad \binom{s}{s_1,\dots,s_t}:=\frac{s!}{s_1! \cdots s_t !}.
\end{align*}
\end{enumerate}
\end{dfn}

\begin{prp} \label{FLi}
We have for $0\leq s \in \ZZ$
\begin{align*}
\mathcal F_s(x)=
\begin{cases}
1 & (s=0), \\
\frac{1}{x+1} \mathrm{Li}_{-s}(\frac{x}{1+x}) & (s\geq 1).
\end{cases}
\end{align*}
\end{prp}

\begin{proof}
This is a well-known identity.
We prove it by using the generating series:
\begin{align}
\sum_{s=0}^\infty \mathrm{Li}_{-s} (x) \frac{y^s}{s!}&=\frac{xe^y}{1-xe^y}, \label{gsofLi} \\
\sum_{s=0}^\infty \mathcal F_s(x) \frac{y^s}{s!}&=\frac{1}{1-x(e^y-1)}. \notag
\end{align}
The former one follows from
\begin{align*}
\left.\left(\frac{d}{dy}\right)^s\frac{xe^y}{1-xe^y}\right|_{y=0}=\left.\left(u\frac{d}{du}\right)^s \frac{u}{1-u} \right|_{u=x} \quad (u:=xe^y)
\end{align*}
and the latter one follows since
\begin{align*}
\sum_{s=0}^\infty \mathcal F_s(x) \frac{y^s}{s!}&=\sum_{t=0}^s x^t \sum_{s=t}^\infty \sum_{\substack{s_1,\dots,s_t\geq 1 \\ s_1+\dots+s_t=s}} \binom{s}{s_1,\dots,s_t}\frac{y^s}{s!} 
=\sum_{t=0}^s x^t \sum_{s_1,\dots,s_t\geq 1} \frac{y^{s_1+\dots+s_t}}{s_1! \cdots s_t !}\\
&=\sum_{t=0}^\infty x^t(e^y-1)^t=\frac{1}{1-x(e^y-1)}.
\end{align*}
Now we have
\begin{align*}
1+\sum_{s=1}^\infty \frac{1}{1+x} \mathrm{Li}_{-s}\left(\frac{x}{1+x}\right) \frac{y^s}{s!}=1+\frac{1}{1+x}\left(\frac{xe^y}{1+x-xe^y} -x\right)=\sum_{s=0}^\infty \mathcal F_s(x) \frac{y^s}{s!}
\end{align*}
as desired.
\end{proof}

\begin{prp} \label{tipsforpl}
\begin{enumerate}
\item We have for $0\leq s \in \ZZ$
\begin{align*}
\mathcal F_s(x)=\sum_{t=0}^s \left(\sum_{k=0}^t (-1)^{t-k} \binom{t}{k}k^s\right) x^t.
\end{align*}
\item We have for $k \in \ZZ$
\begin{align*}
\sum_{\substack{1\leq j,l \leq q-1 \\ jl\equiv k \mod q}} jl \equiv k^q \sum_{j=1}^{q-1} j j^* \mod q^2.
\end{align*}
\item We have
\begin{align*}
f_q(-x) \equiv c_q \int_0^x \mathcal F_q(x)\frac{dx}{x} \mod q^2 \quad \text{for} \quad c_q:=\sum_{j=1}^{q-1} jj^*.
\end{align*}
\end{enumerate}
\end{prp}

\begin{proof} 
(1) It is clear since the coefficients of $x^t$ on both sides count the same number of
surjective maps of $\{1,\dots,s\}$ to $\{1,\dots,t\}$.

\medskip

\noindent
(2) When $q \mid k$, it is trivial since both sides are congruent to $0$.
We assume that $q \nmid k$.
The condition $jl\equiv k \mod q$ holds if and only if $l\equiv j^*k^q \mod q$.
Hence, under the assumption $1\leq l \leq q-1$, such $l$ is unique and given by $l=q(\frac{j^*k^q}{q}-[\frac{j^*k^q}{q}])$,
where we denote the integer part of $a \in \QQ$ by $[a]$.
Therefore we have
\begin{align} \label{ip1}
\sum_{\substack{1\leq j,l \leq q-1 \\ jl\equiv k \mod q}} jl - k^q \sum_{j=1}^{q-1} j j^*=-q\sum_{j=1}^{q-1}j\left[\frac{j^*k^q}{q}\right].
\end{align}
On the other hand, by noting that $[\frac{j^*k^q}{q}]= [\frac{j^*}{q}(k+k^q-k) ]= [\frac{j^*k}{q}]+\frac{k^q-k}{q}\cdot j^*$, we have
\begin{align} \label{ip2}
\sum_{j=1}^{q-1}j\left[\frac{j^*k^q}{q}\right]=\sum_{j=1}^{q-1} j \left[\frac{j^*k}{q}\right]+\frac{k^q-k}{q} \sum_{j=1}^{q-1} jj^*
\equiv \sum_{j=1}^{q-1} \frac{1}{j^*} \left[\frac{j^*k}{q}\right]-\frac{k^q-k}{q} \mod q.
\end{align}
Then the assertion follows from \cref{ip1,ip2} and Lerch's congruence \cite[($8^*$)]{Ler} 
\begin{align*} 
\frac{k^q-k}{q} \equiv \sum_{j=1}^{q-1} \frac{1}{j} \left[\frac{jk}{q}\right] \mod q.
\end{align*}

\noindent
(3) follows immediately from (1), (2), and the definition of $f_q(x)$.
\end{proof}

\begin{thm} \label{pl}
\begin{enumerate}
\item We have
\begin{align*}
f_q(x) \equiv c_q \mathrm{Li}_{1-q}\left(\frac{-x}{1-x}\right) \mod q^2,
\end{align*}
where $c_q=\sum_{j=1}^{q-1} j j^*$.
\item Let $m,n \in \ZZ$, and assume that $m-n \not \equiv 0 \mod q$.
The following are equivalent.
\begin{enumerate}
\item $\mathrm{Li}_{1-q}(\frac{n}{m}) \in q^2\ZZ_{(q)}$.
\item $f_q(\frac{-n}{m-n}) \in q^2\ZZ_{(q)}$.
\end{enumerate}
\item We have 
\begin{align} 
\mathrm{Li}_{1-q}(x^{-1})&=-\mathrm{Li}_{1-q}(x), \label{symofLi} \\
f_q(1-x)&\equiv -f_q(x) \mod q^2. \label{sym}
\end{align}
In particular we obtain
\begin{enumerate}
\item $\mathrm{Li}_{1-q}(b) \in q^2\ZZ_{(q)}$ if and only if $\mathrm{Li}_{1-q}(b^{-1}) \in q^2\ZZ_{(q)}$ for $b \in \QQ$ with $b \neq 0,1$.
\item $b \bmod q^2 \in S_q$ if and only if $1-b \bmod q^2 \in S_q$ for $b \in \ZZ_{(q)}$.
\item $\mathrm{Li}_{1-q}(\frac{n}{m}) \in q^2\ZZ_{(q)}$ when $\frac{n}{m} \in q^2\ZZ_{(q)}$, $\frac{n}{m} \in -1+q^2\ZZ_{(q)}$ or $\frac{m}{n} \in q^2\ZZ_{(q)}$.
\item $0,\frac{1}{2},1 \bmod q^2 \in S_q$.
\end{enumerate}
\end{enumerate}
\end{thm}

\begin{proof}
(1) From \cref{FLi} we obtain
\begin{align*}
\int_0^x \mathcal F_q(x) \frac{dx}{x}=\int_0^x \mathrm{Li}_{-q}\left(\frac{x}{1+x}\right) \frac{dx}{x(1+x)}
=\int_0^{\frac{x}{1+x}} \mathrm{Li}_{-q}(z) \frac{dz}{z}=\mathrm{Li}_{1-q}\left(\frac{x}{1+x}\right).
\end{align*}
Then the assertion follows by combining it with \cref{tipsforpl}-(3).

\medskip

\noindent
(2) follows from (1) and $c_q \equiv q-1 \not \equiv 0 \mod q$.

\medskip

\noindent
(3) The generating series \cref{gsofLi} of $\mathrm{Li}_{-s}(x)$ satisfies the following symmetry.
\begin{align*}
\frac{x^{-1}e^{-y}}{1-x^{-1}e^{-y}}=\frac{-1}{1-xe^y}=-1-\frac{xe^y}{1-xe^y}.
\end{align*}
It implies the functional equation (\ref{symofLi}).
\cref{sym} follows from this and (1).
Then (a) and (b) follow immediately.
For (c), the case $\frac{n}{m} \in q^2\ZZ_{(q)}$ follows from $x \mid \mathrm{Li}_{q-1}(x)$.
Hence the case $\frac{m}{n} \in q^2\ZZ_{(q)}$ follows by (a).
The case $\frac{n}{m} \in -1+q^2\ZZ_{(q)}$ follows by substituting $x=-1$ in \cref{symofLi}.
(d) follows from (c) and (2).
\end{proof}

\begin{rmk} \label{finrmk}
If we exchange $m$ and $n$ then $\frac{-n}{m-n}$ becomes $1-\frac{-n}{m-n}$, although $p=\Phi_q(n/m)m^{q-1}$ remains unchanged.
The assertions of \cref{pl}-(3) correspond to this symmetry. 
\end{rmk}

\bibliographystyle{alpha} 
\bibliography{reclaw} 

\begin{thebibliography}{Lem00}

\bibitem[Dic35]{Di}
L.~E. Dickson.
\newblock Cyclotomy, {H}igher {C}ongruences, and {W}aring's {P}roblem.
\newblock {\em Amer. J. Math.}, 57(2):391--424, 1935.

\bibitem[FT62]{FT}
W.~Feit and J.~G. Thompson.
\newblock A solvability criterion for finite groups and some consequences.
\newblock {\em Proc. Nat. Acad. Sci. U.S.A.}, 48:968--970, 1962.

\bibitem[Gra03]{Gr}
G.~Gras.
\newblock {\em Class field theory: From theory to practice}.
\newblock Springer Monographs in Mathematics. Springer-Verlag, Berlin, 2003.
\newblock Translated from the French manuscript by Henri Cohen.

\bibitem[Kob84]{Ko}
N.~Koblitz.
\newblock {\em {$p$}-adic numbers, {$p$}-adic analysis, and zeta-functions},
  volume~58 of {\em Graduate Texts in Mathematics}.
\newblock Springer-Verlag, New York, second edition, 1984.

\bibitem[Le12]{Le}
M.~Le.
\newblock A divisibility problem concerning group theory.
\newblock {\em Pure Appl. Math. Q.}, 8(3):689--691, 2012.

\bibitem[Leh66]{Leh}
E.~Lehmer.
\newblock Artiads characterized.
\newblock {\em J. Math. Anal. Appl.}, 15:118--131, 1966.

\bibitem[Lem00]{Lem}
F.~Lemmermeyer.
\newblock {\em Reciprocity laws: From Euler to Eisenstein}.
\newblock Springer Monographs in Mathematics. Springer-Verlag, Berlin, 2000.

\bibitem[Ler05]{Ler}
M.~Lerch.
\newblock Zur {T}heorie des {F}ermatschen {Q}uotienten
  {$\frac{{a^{p-1}-1}}{p}=q(a)$}.
\newblock {\em Math. Ann.}, 60(4):471--490, 1905.

\bibitem[Mus64]{Mu}
J.~B. Muskat.
\newblock On the solvability of {$x\sp{e}\equiv e\,({\rm mod}\,p)$}.
\newblock {\em Pacific J. Math.}, 14:257--260, 1964.

\end{thebibliography}

\end{document}